\documentclass[12pt]{amsart}   

\usepackage{graphicx}
\usepackage{tikz}
\usepackage{amssymb}
\usetikzlibrary{calc}
\usepackage{url}
\usepackage{caption}
\usepackage{tcolorbox}
\usepackage{listings}

\newtheorem{thm}{Theorem}[section]

\theoremstyle{definition}

\theoremstyle{remark}

\begin{document}

%\title{Automated Conjecturing VII: \\The Graph Brain Project \\ \& Big Mathematics
%}
\title{(Avoiding) Proof by Contradiction:\\  $\sqrt{2}$ is Not Rational
}

%\author{N. Bushaw, C. E. Larson$^*$, N. Van Cleemput$^{**}$, \\and other workshop participants}
%\footnote{Corresponding author: \url{clarson@vcu.edu}}
%\author{L. Hutchinson, V. Kamat, C. E. Larson\footnote{test}, S. Mehta, D. Muncy, N. Van Cleemput$^*$}
\author{C. E. Larson$^*$}
%math1um, rbarden, cjoshea9, thenealon, PratipRana, nvcleemp,  alf3x,
%shivejl, bigmath, wardbp, neogeo44, wilcoxcookwn (neil wilcox-cook),
%holdenig

%david muncy, brandon harris, corrine, joseph raines

\address{Department of Mathematics and Applied Mathematics\\Virginia Commonwealth University\\Richmond, VA 23284, USA }

%\ead{clarson@vcu.edu}
\thanks{(*) Research supported by the Simons Foundation Mathematics and Physical Sciences--Collaboration Grants for Mathematicians Award (426267)}


\date{}

\maketitle

\begin{abstract}
We provide an alternative proof that $\sqrt{2}$ is irrational that does not begin with the assumption that $\sqrt{2}$ is in fact rational.
\end{abstract}

While all professional mathematicians use \textit{proof by contradiction} as a tool of the trade many encounter resistance to the introduction of this technique to mathematical initiates. In particular, in the standard classroom proof, a student may wonder: how can it be that the teacher has both assumed that $\sqrt{2}$ is rational and proved that it is not?!?! What the teacher has done may provoke cognitive dissonance---and the student may resist what the professional has become \textit{used to}.

The two most standard proofs by contradiction in the undergraduate curriculum are the proof that there are infinitely many primes and the proof that $\sqrt{2}$ is irrational (see, for example \cite{Hamm13}).
It is well-known that the traditional proof by contradiction for the infinitude of primes can easily be converted into a proof for the construction of an unending sequence of primes: there is no need to ever make the assumption that there are finitely many primes---and no need to ever produce a contradiction. Here we provide an analogous proof that the square root of two is irrational. We don't begin by assuming that that the $\sqrt{2}$ is rational: rather we show that, beginning with \text{any} rational approximation of the $\sqrt{2}$, we can produce an unending sequence of better approximations.

Before beginning we remark that, just as the reworked proof of the infinitude of primes contains the main idea of the traditional proof by contradiction,  the main idea here (parity) is repurposed from the traditional proof. It is also worth mentioning that the production of proofs that that $\sqrt{2}$ is irrational has continued to the present day (\cite{Half55,GaunRabs56,Fine76,Floy89,Bloo95,Beig91,KalmMenaShah97,Apos00,Unga06,Berr08,Ferr09,Cair12,MoreGarc13,Roun19}), that interesting discussions exist,  and that philosophical issues surrounding proof by contradiction by intuitionists and constructivists have also generated heat for at least 100 years (\cite{Bish85}). 

\begin{thm}
If $a$ and $b$ are positive integers with $\frac{a}{b}\geq \sqrt{2}$ (that is, $a^2\geq 2b^2$) then there are positive integers $a'',b''$ such that:
\[
\frac{a}{b}>\frac{a''}{b''}\geq \sqrt{2}.
\]
\end{thm}

\begin{proof}
Let $a$ and $b$ be positive integers with $a^2\geq 2b^2$. If $a$ and $b$ are both even let $2^k$ be the largest power of $2$ dividing each. Let $a'=\frac{a}{2^k}$ and $b'=\frac{b}{2^k}$; then $a'$ and $b'$ are not both even, $(a'2^k)^2\geq 2(b'2^k)^2$ and 
$a'^2\geq 2b'^2$. If $a'$ is odd, then $a'^2$ 
is odd and it must be that $a'^2>2b'^2$. If $a'$ is even (so $a'=2c$ for some integer $c$) and $b'$ is odd (so $b'^2$ is odd), then $(2c)^2=a'^2\geq 2b'^2$. As $2c^2\geq b'^2$, and $2c^2$ is even and $b'^2$ is odd, we have $2c^2>b'^2$, and then $a'^2>2b'^2$. So in either case  $a'^2>2b'^2$. 

% Since $a'$ and $b'$ are not both even it cannot be the case that $a'^2= 2b'^2$: if $a'$ is odd then $a'^2$ is odd, and if $a'=2c So $a'^2> 2b'^2$.

Now let  $a''=a'^2+2b'^2$ and $b''=2a'b'$. Then:
\[
a'^2>2b'^2,
\]
\[
2a'^2>a'^2+2b'^2,
\]
\[
a'(2a'b')>(a'^2+2b'^2)b',
\]
\[
a'b''>a''b',
\]
\[
\frac{a'}{b'}>\frac{a''}{b''}.
\]
\[
\text{And}, \text{ as } \frac{a}{b}=\frac{a'}{b'}, \frac{a}{b}>\frac{a''}{b''}.
\]

Also:
\[
(a'^2-2b'^2)^2\geq 0,
\]
\[
a'^4-4a'^2b'^2+4b'^4\geq 0
\]
\[
a'^4+4a'^2b'^2+4b'^4\geq 8a'^2b'^2,
\]
\[
(a'^2+2b'^2)^2\geq 2(2a'b')^2,
\]
\[
a''^2\geq 2b''^2,
\]
\[
\frac{a''}{b''}\geq \sqrt{2}.
\]

\end{proof}

So we have an algorithm (a variation of a possible Babylonian algorithm \cite{FowlRobs98}) for producing better and better approximations to $\sqrt{2}$ given \textit{any} rational number $\frac{a}{b}$ with $\frac{a}{b}\geq \sqrt{2}$. And it can't then be the case that $\sqrt{2}$ is a rational number.\\

\textbf{Acknowledgement.} A great many of the citations here originate from A. Bogomolny's collection of proofs at the Cut the Knot website (\url{http://www.cut-the-knot.org/proofs/sq_root.shtml}).

\bibliographystyle{plain}
\bibliography{../../larson.bib}

\def\cprime{$'$} \def\cprime{$'$} \def\cprime{$'$} \def\cprime{$'$}
\begin{thebibliography}{10}

\bibitem{Apos00}
T.~M. Apostol.
\newblock Irrationality of the square root of two-a geometric proof.
\newblock {\em American Mathematical Monthly}, 107(9):841--841, 2000.

\bibitem{Beig91}
R.~Beigel.
\newblock Irrationality without number theory.
\newblock {\em The American Mathematical Monthly}, 98(4):332--335, 1991.

\bibitem{Berr08}
G.~Berresford.
\newblock A simpler proof of a well-known fact.
\newblock {\em Am. Math. Mon}, 115:524, 2008.

\bibitem{Bish85}
E.~Bishop.
\newblock Schizophrenia in contemporary mathematics.
\newblock In {\em Errett Bishop: Reflections on Him and His Research},
  volume~39 of {\em Contemporary Mathematics}, pages 1--32. American
  Mathematical Society, 1985.

\bibitem{Bloo95}
D.~M. Bloom.
\newblock A one-sentence proof that $\sqrt{2}$ is irrational.
\newblock {\em Mathematics Magazine}, 68(4):286--286, 1995.

\bibitem{Cair12}
G.~Cairns.
\newblock Proof without words: $\sqrt{2}$ is irrational.
\newblock {\em Mathematics Magazine}, 85(2):123, 2012.

\bibitem{Ferr09}
N.~C. Ferre{\~n}o.
\newblock Yet another proof of the irrationality of $\sqrt{2}$.
\newblock {\em The American Mathematical Monthly}, 116(1):68--69, 2009.

\bibitem{Fine76}
N.~Fine.
\newblock Look, ma, no primes.
\newblock {\em Mathematics Magazine}, 49(5):249--250, 1976.

\bibitem{Floy89}
R.~Floyd.
\newblock What else pythagoras could have done.
\newblock {\em American Mathematical Monthly}, 96(1):67--67, 1989.

\bibitem{FowlRobs98}
D.~Fowler and E.~Robson.
\newblock Square root approximations in old babylonian mathematics: Ybc 7289 in
  context.
\newblock {\em Historia Mathematica}, 25(4):366--378, 1998.

\bibitem{GaunRabs56}
R.~Gauntt and R.~Rabson.
\newblock The irrationality of $\sqrt{2}$.
\newblock {\em The American Mathematical Monthly}, 63(4):247--247, 1956.

\bibitem{Half55}
E.~Halfar.
\newblock The irrationality of $\sqrt{2}$.
\newblock {\em The American Mathematical Monthly}, 62(6):437--437, 1955.

\bibitem{Hamm13}
R.~Hammack.
\newblock {\em The Book of Proof, 3rd ed.}
\newblock Richard Hammack, 2013.

\bibitem{KalmMenaShah97}
D.~Kalman, R.~Mena, and S.~Shahriari.
\newblock Variations on an irrational theme—geometry, dynamics, algebra.
\newblock {\em Mathematics Magazine}, 70(2):93--104, 1997.

\bibitem{MoreGarc13}
S.~G. Moreno and E.~M. Garc{\'\i}a-Caballero.
\newblock On the irrationality of $\sqrt{2}$ once again.
\newblock {\em The American Mathematical Monthly}, 120(7):674--674, 2013.

\bibitem{Roun19}
B.~Rounds.
\newblock Euclid's lemma and the square root of 2.
\newblock {\em The American Mathematical Monthly}, 126(3):274--274, 2019.

\bibitem{Unga06}
P.~Ungar.
\newblock Irrationality of square roots.
\newblock {\em Mathematics Magazine}, 79(2):147, 2006.

\end{thebibliography}
%
%\begin{thebibliography}{}
%%
%% and use \bibitem to create references. Consult the Instructions
%% for authors for reference list style.
%\end{thebibliography}

\end{document}